\documentclass[11pt]{article}

\usepackage{epsfig,epsf,fancybox}
\usepackage{amsmath}
\usepackage{mathrsfs}
\usepackage{amssymb}
\usepackage{graphicx}
\usepackage{subfigure}
\usepackage{color}
\usepackage{algorithm}
\usepackage{algorithmicx}
\usepackage{algpseudocode}

\newtheorem{theorem}{Theorem}[section]

\newtheorem{lemma}[theorem]{Lemma}

\newtheorem{definition}{Definition}[section]

\newtheorem{remark}{Remark}
\numberwithin{equation}{section}

\newcommand{\st}{\textnormal{s.t.}}

\newcommand{\Prob}{\textnormal{Prob}}

\newcommand{\sign}{\textnormal{sign}}

\newcommand{\be}{\begin{equation}}
\newcommand{\ee}{\end{equation}}
\newcommand{\ba}{\begin{array}}
\newcommand{\ea}{\end{array}}
\newcommand{\half}{\frac{1}{2}}
\newcommand{\br}{\mathbb{R}}

\newcommand{\mtn}{{m \times n}}
\newcommand{\LCal}{\mathcal{L}}

\newcommand{\XCal}{\mathcal{X}}
\newcommand{\YCal}{\mathcal{Y}}
\newcommand{\ZCal}{\mathcal{Z}}

\newcommand{\Shrink}{\mathrm{Shrink}}

\newcommand{\etal}{{\it et al.\ }}

\newcommand{\argmin}{\mathrm{argmin}}
\newcommand{\by}{\bar{y}}
\newcommand{\blambda}{\bar{\lambda}}
\newcommand{\bz}{\bar{z}}

\newcommand{\bpm}{\begin{pmatrix}}
\newcommand{\epm}{\end{pmatrix}}

\begin{document}

\title{An Extragradient-Based Alternating Direction Method \\ for Convex Minimization}

\author{
Tianyi Lin
\thanks{Department of Systems Engineering and Engineering Management, The Chinese University of Hong Kong, Shatin, N. T., Hong Kong. Email: linty@se.cuhk.edu.hk.}
\and Shiqian Ma
\thanks{Department of Systems Engineering and Engineering Management, The Chinese University of Hong Kong, Shatin, N. T., Hong Kong. Email: sqma@se.cuhk.edu.hk. Research of this author was supported in part by a Direct Grant of The Chinese University of Hong
Kong (Project ID: 4055016) and the Hong Kong Research Grants Council
General Research Funds Early Career Scheme (Project ID: CUHK 439513).}
\and Shuzhong Zhang
\thanks{Department of Industrial and Systems Engineering, University of Minnesota, Minneapolis, MN 55455. Email: zhangs@umn.edu. Research of this author was supported in part by the NSF Grant CMMI-1161242.}}

\date{\today}
\maketitle

\begin{abstract}
In this paper, we consider the problem of minimizing the sum of two convex functions subject to linear linking constraints. The classical alternating direction type methods usually assume that the two convex functions have relatively easy proximal mappings. However, many problems arising from statistics, image processing and other fields have the structure that while one of the two functions has easy proximal mapping, the other function is smoothly convex but does not have an easy proximal mapping. Therefore, the classical alternating direction methods cannot be applied. To deal with the difficulty,
we propose in this paper an alternating direction method based on {\em extragradients}. Under the assumption that the smooth function has a Lipschitz continuous gradient, we prove that the proposed method returns an $\epsilon$-optimal solution within $O(1/\epsilon)$ iterations. We apply the proposed method to solve a new statistical model called fused logistic regression. Our numerical experiments show that the proposed method performs very well when solving the test problems. We also test the performance of the proposed method through solving the lasso problem arising from statistics and compare the result with several existing efficient solvers for this problem; the results are very encouraging indeed.

\vspace{0.3cm}

\vspace{0.3cm}


\end{abstract}

\newpage

\section{Introduction} \label{sec:intro}

In this paper, we consider solving the following convex optimization problem:
\be\label{prob:min-sum-2}\ba{ll} \min_{x\in\br^n,y\in\br^p} & f(x) + g(y) \\ \st & Ax+ By =b \\ & x\in\XCal, y\in\YCal, \ea\ee
where $f$ and $g$ are convex functions, $A\in\br^{m\times n}$, $B\in\br^{m\times p}$, $b\in\br^m$, $\XCal$ and $\YCal$ are convex sets and the projections on them can be easily obtained. Problems in the form of \eqref{prob:min-sum-2} arise in different applications in practice and we will show some examples later.
A recent very popular way to solve \eqref{prob:min-sum-2} is to apply the alternating direction method of multipliers (ADMM). A typical iteration of ADMM for solving \eqref{prob:min-sum-2} can be described as
\be\label{alg:admm-original}\left\{\ba{lll} x^{k+1} & := & \argmin_{x\in \XCal} \ \LCal_\gamma(x,y^k;\lambda^k) \\
                                            y^{k+1} & := & \argmin_{y\in \YCal} \ \LCal_\gamma(x^{k+1},y;\lambda^k) \\
                                            \lambda^{k+1} & := & \lambda^k - \gamma(Ax^{k+1}+By^{k+1}-b),\ea\right.\ee
where the augmented Lagrangian function $\LCal_\gamma(x,y;\lambda)$ for \eqref{prob:min-sum-2} is defined as
\be\label{aug-lag-func} \LCal_\gamma(x,y;\lambda) := f(x) + g(y) -\langle\lambda,Ax+By-b\rangle+\frac{\gamma}{2}\|Ax+By-b\|^2, \ee
where $\lambda$ is the Lagrange multiplier associated with the linear constraint $Ax+By=b$ and $\gamma>0$ is a penalty parameter. The ADMM is closely related to some operator splitting methods such as Douglas-Rachford operator splitting method \cite{Douglas-Rachford-56} and Peaceman-Rachford operator splitting method \cite{Peaceman-Rachford-55} for finding the zero of the sum of two maximal monotone operators. In particular, it was shown by Gabay \cite{Gabay-83} that ADMM \eqref{alg:admm-original} is equivalent to applying the Douglas-Rachford operator splitting method to the dual problem of \eqref{prob:min-sum-2}. The ADMM and operator splitting methods were then studied extensively in the literature and some generalized variants were proposed (see, e.g., \cite{Lions-Mercier-79,Fortin-Glowinski-1983,Glowinski-LeTallec-89,Eckstein-thesis-89,Eckstein-Bertsekas-1992}). The ADMM was revisited recently because it was found very efficient for solving many sparse and low-rank optimization problems, such as compressed sensing \cite{Yang-Zhang-2009}, compressive imaging \cite{Wang-Yang-Yin-Zhang-2008,Goldstein-Osher-08}, robust PCA \cite{Tao-Yuan-SPCP-2011}, sparse inverse covariance selection \cite{Yuan-2009,Scheinberg-Ma-Goldfarb-NIPS-2010}, sparse PCA \cite{Ma-SPCA-2011-submit} and semidefinite programming \cite{Wen-Goldfarb-Yin-2009} etc. Moreover, the iteration complexity of ADMM \eqref{alg:admm-original} was recently established by He and Yuan \cite{He-Yuan-rate-ADM-2012} and Monteiro and Svaiter \cite{Monteiro-Svaiter-2010a}, and some new analysis for iteration complexity for obtaining $\epsilon$-optimal solution measured by both objective error and constraint violation was given in \cite{Lin-Ma-Zhang-admm-sublinear-2014} and \cite{Davis-Yin-2014a}. The recent survey paper by Boyd \etal \cite{Boyd-etal-ADM-survey-2011} listed many interesting applications of ADMM in statistical learning and distributed optimization.

Note that the efficiency of ADMM \eqref{alg:admm-original} actually depends on whether the two subproblems in \eqref{alg:admm-original} can be solved efficiently or not. This requires that the following two problems can be solved efficiently for given $\tau>0$, $w,z\in\br^m$:
\be\label{sub-f} x := \argmin_{x\in\XCal} \ f(x) + \frac{1}{2\tau}\|Ax-w\|^2 \ee
and
\be\label{sub-g} y := \argmin_{y\in\YCal} \ g(y) + \frac{1}{2\tau}\|By-z\|^2. \ee
When $\XCal$ and $\YCal$ are the whole spaces and $A$ and $B$ are identity matrices, \eqref{sub-f} and \eqref{sub-g} are known as the proximal mappings of functions $f$ and $g$, respectively. Thus, in this case, ADMM \eqref{alg:admm-original} requires that the proximal mappings of $f$ and $g$ are easy to be obtained. In the cases that $A$ and $B$ are not identity matrices, there are results on linearized ADMM (see, e.g., \cite{Yang-Zhang-2009,Yang-Yuan-Linearized-trace-norm,Zhang-Burger-Bresson-Osher-09}) which linearizes the quadratic penalty term in such a way that problems \eqref{sub-f} and \eqref{sub-g} still correspond to the proximal mappings of functions $f$ and $g$. The global convergence of the linearized ADMM is guaranteed under certain conditions on a linearization step size parameter.

There are two interesting problems that are readily solved by ADMM \eqref{alg:admm-original} since the involved functions have easy proximal mappings.
One problem is the so-called robust principal component pursuit (RPCP) problem:
\be\label{prob:rpcp} \min_{X\in\br^\mtn,Y\in\br^\mtn} \ \|X\|_* + \rho\|Y\|_1, \ \st, \ X + Y =M, \ee
where $\rho>0$ is a weighting parameter, $M\in\br^\mtn$ is a given matrix, $\|X\|_*$ is the nuclear norm of $X$, which is defined as the sum of singular values of $X$, and $\|Y\|_1:=\sum_{i,j}|Y_{ij}|$ is the $\ell_1$ norm of $Y$.
Problem \eqref{prob:rpcp} was studied by Cand\`es \etal \cite{Candes-Li-Ma-Wright-RPCA-2009} and Chandrasekaran \etal \cite{Chandrasekaran-Sanghavi-Parrilo-Willsky-2009} as a convex relaxation of the robust PCA problem. Note that the two involved functions, the nuclear norm $\|X\|_*$ and the $\ell_1$ norm $\|Y\|_1$, have easy proximal mappings (see, e.g., \cite{Ma-Goldfarb-Chen-2008} and \cite{Hale-Yin-Zhang-SIAM-2008}).
The other problem is the so-called sparse inverse covariance selection, which is also known as the graphical lasso problem \cite{Yuan-Lin-2007,Banerjee-ElGhaoui-Aspremont-2008,Friedman-Hastie-Tibshirani-2007}. This problem, which estimates a sparse inverse covariance matrix from sample data, can be formulated as
\be\label{prob:glasso} \min_{X\in\br^{n\times n}} \ -\log\det(X) + \langle \hat{\Sigma},X \rangle + \rho\|X\|_1, \ee
where the first convex function $-\log\det(X) + \langle \hat{\Sigma},X \rangle$ is the negative log-likelihood function for given sample covariance matrix $\hat{\Sigma}$, and the second convex function $\rho\|X\|_1$ is used to promote the sparsity of the resulting solution. Problem \eqref{prob:glasso} is of the form of \eqref{prob:min-sum-2} because it can be rewritten equivalently as
\be\label{prob:glasso-XY} \min_{X\in\br^{n\times n},Y\in\br^{n\times n}} \ -\log\det(X) + \langle \hat{\Sigma},X \rangle + \rho\|Y\|_1, \ \st, \ X-Y=0. \ee
Note that the involved function $-\log\det(X)$ has an easy proximal mapping (see, e.g., \cite{Scheinberg-Ma-Goldfarb-NIPS-2010} and \cite{Yuan-2009}).

However, there are many problems arising from statistics, machine learning and image processing which do not have easy subproblems \eqref{sub-f} and \eqref{sub-g} even when $A$ and $B$ are identity matrices. One such example is the so-called sparse logistic regression problem. For given training set $\{a_i,b_i\}_{i=1}^m$ where $a_1,a_2,\ldots,a_m$ are the $m$ samples and $b_1,\ldots,b_m$ with $b_i\in\{-1,+1\},i=1,\ldots,m$ are the binary class labels. The likelihood function for these $m$ samples is $\prod_{i=1}^m \Prob(b_i\mid a_i)$, where
\[\Prob(b\mid a) := \frac{1}{1+\exp(-b(a^\top x+c))},\]
is the conditional probability of the label $b$ condition on sample $a$, where $x\in\br^n$ is the weight vector and $c\in\br$ is the intercept, and $a^\top x+c=0$ defines a hyperplane in the feature space, on which $\Prob(b\mid a)=0.5$. Besides, $\Prob(b\mid a)>0.5$ if $a^\top x+c$ has the same sign as $b$, and $\Prob(b\mid a)<0.5$ otherwise. The sparse logistic regression (see \cite{Liu-logistic-regression-2009}) can be formulated as the following convex optimization problem
\be\label{logistic-reg} \min_{x,c} \ \ell(x,c) + \alpha \|x\|_1, \ee
where $\alpha>0$ is a weighting parameter, $\ell(x,c)$ denotes the average logistic loss function, which is defined as
\[\ell(x,c) := -\frac{1}{m}\prod_{i=1}^m \Prob(b_i\mid a_i) = \frac{1}{m}\sum_{i=1}^m \log(1+\exp(-b_i(a_i^\top x+c))),\]
and the $\ell_1$ norm $\|x\|_1$ is imposed to promote the sparsity of the weight vector $x$.
If one wants to apply ADMM \eqref{alg:admm-original} to solve \eqref{logistic-reg}, one has to introduce a new variable $y$ and rewrite \eqref{logistic-reg} as
\be\label{logistic-reg-xy}\ba{ll} \min_{x,c,y} & \ell(x,c) + \alpha \|y\|_1, \\
                                \st &  x - y =0.\ea\ee
When ADMM \eqref{alg:admm-original} is applied to solve \eqref{logistic-reg-xy}, although the subproblem with respect to $y$ is easily solvable (an $\ell_1$ shrinkage operation), the subproblem with respect to $(x,c)$ is difficult to solve because the proximal mapping of the logistic loss function $\ell(x,c)$ is not easily computable.

Another example is the following fused logistic regression problem:
\be\label{fused-logistic-reg-intro}\min_{x,c} \ \ell(x,c) + \alpha\|x\|_1 + \beta\sum_{j=2}^n |x_j-x_{j-1}|, \ee
or its constrained version:
\[\ba{ll} \min_{x,c} & \ell(x,c) + \alpha\|x\|_1 \\
            \st      & \sum_{j=2}^n |x_j-x_{j-1}| \leq \xi,\ea\]
            where $\alpha$, $\beta$ and $\xi$ are positive parameters.
These problems cannot be solved by ADMM \eqref{alg:admm-original}, again because of the difficulty in
computing the proximal mapping of $\ell(x,c)$. We will discuss this example in more details in Section \ref{sec:fused-logistic}.

But is it really crucial to compute the proximal mapping exactly in the ADMM scheme? After all, ADMM can be viewed as an approximate dual gradient ascent method. As such, computing the proximal mapping exactly is in some sense redundant, since on the dual side the iterates are updated based on the gradient ascent method. Without sacrificing the scale of approximation to optimality, an update on the primal side based on the gradient information (or at least part of it), by the principle of primal-dual symmetry, is entirely appropriate. Our subsequent analysis indeed confirms this belief. Moreover, there are inexact versions of augmented Lagrangian method and alternating direction method (see, e.g. \cite{He-Liao-Han-Yang-2002,Eckstein-Silva-inexact-alm-2013}), and our proposed method can be viewed as a variant of the inexact version of ADMM.

{\bf Our contribution}. In this paper, we propose a new alternating direction method for solving \eqref{prob:min-sum-2}. This new method requires only one of the functions in the objective to have an easy proximal mapping, and the other involved function is merely required to be smooth. Note that the aforementioned examples, namely sparse logistic regression \eqref{logistic-reg} and fused logistic regression \eqref{fused-logistic-reg-intro}, are both of this type. In each iteration, the proposed method involves only computing the proximal mapping for one function and computing the gradient for the other function. Under the assumption that the smooth function has a Lipschitz continuous gradient, we prove that the proposed method finds an $\epsilon$-optimal solution to \eqref{prob:min-sum-2} within $O(1/\epsilon)$ iterations.

The rest of this paper is organized as follows. In Section \ref{sec:egadm} we propose the extragradient-based alternating direction method for solving problem \eqref{prob:min-sum-2}. The iteration complexity of the proposed method is analyzed in Section \ref{sec:complexity}. In Section \ref{sec:fused-logistic}, we discuss the details of the fused logistic regression problem and how to use our proposed method to solve it. Numerical experiments are conducted in Section \ref{sec:numerical}. Finally we draw some conclusions in Section \ref{sec:conclusion}.

\section{An Alternating Direction Method Based on Extragradient}\label{sec:egadm}

In this section, we consider solving \eqref{prob:min-sum-2} where $f$ has an easy proximal mapping, while $g$ is smooth but does not have an easy proximal mapping. Note that in this case, ADMM \eqref{alg:admm-original} cannot be applied to solve \eqref{prob:min-sum-2} as the solution for the second subproblem in \eqref{alg:admm-original} is not available.

Note that in ADMM \eqref{alg:admm-original}, the updating formula for $\lambda$ can be seen as a gradient step for $\lambda$ with respect to the augmented Lagrangian function.
Now because minimizing the augmented Lagrangian function with respect to $y$ in \eqref{alg:admm-original} is not possible, we can also take a gradient step. However, our analysis shown later indicates that we have to take a gradient step for the Lagrangian function
\be\label{lag-func} \LCal(x,y;\lambda) := f(x) + g(y) - \langle \lambda, Ax+By-b \rangle\ee
with respect to $y$.
This leads to the following updating procedure for $(x^k,y^k,\lambda^k)$:
\be\label{alg:g-adm}\left\{\ba{lll}
x^{k+1} & := & \argmin_{x\in\XCal} \ \LCal_{\gamma}(x,y^k;\lambda^k) + \half\|x-x^k\|_{H}^2\\
y^{k+1} & := & [y^k - \gamma \nabla_y \LCal(x^{k+1},y^k;\lambda^k)]_{\YCal} \\
\lambda^{k+1} & := & \lambda^k - \gamma (Ax^{k+1}+By^{k+1}-b),
\ea\right.\ee
where $[y]_\YCal$ denotes the projection of $y$ onto $\YCal$. Note that we have added a proximal term $\half\|x-x^k\|_{H}^2$ to the $x$-subproblem, where $H$ is a pre-specified positive semidefinite matrix.
In this paper, we propose to take an extra gradient step for both $y$ and $\lambda$, which results in the following extragradient-based alternating direction method (EGADM) to solve Problem \eqref{prob:min-sum-2}. Starting from
any initial point $x^0\in\XCal$, $y^0\in\YCal$ and $\lambda^0\in\br^m$, a typical iteration of EGADM can be described as:
\be\label{alg:eg-adm}\left\{\ba{lll}
x^{k+1} & := & \argmin_{x\in\XCal} \ \LCal_{\gamma}(x,y^k;\lambda^k) + \half\|x-x^k\|_{H}^2\\
\bar{y}^{k+1} & := & [y^k - \gamma \nabla_y \LCal(x^{k+1},y^k;\lambda^k)]_{\YCal} \\
\bar{\lambda}^{k+1} & := & \lambda^k - \gamma (Ax^{k+1}+By^k-b) \\
y^{k+1} & := & [y^k - \gamma \nabla_y \LCal(x^{k+1},\bar{y}^{k+1};\bar{\lambda}^{k+1})]_{\YCal} \\
\lambda^{k+1} & := & \lambda^k - \gamma (Ax^{k+1}+B\bar{y}^{k+1}-b).
\ea\right.\ee

Note that the first subproblem in \eqref{alg:eg-adm} is to minimize the augmented Lagrangian function plus a proximal term $\half\|x-x^k\|_{H}^2$ with respect to $x$, i.e.,
\be\label{alg:eg-adm-sub-x} x^{k+1} := \argmin_{x\in\XCal} \ f(x) -\langle\lambda^k,Ax+By^k-b\rangle+\frac{\gamma}{2}\|Ax+By^k-b\|^2 + \half\|x-x^k\|_{H}^2.\ee
In sparse and low-rank optimization problems, the proximal term $\half\|x-x^k\|_{H}^2$ is usually imposed to cancel the effect of matrix $A$ in the quadratic penalty term. One typical choice of $H$ is $H=0$ when $A$ is identity, and $H=\tau I - \gamma A^\top A$ when $A$ is not identity, where $\tau > \gamma \lambda_{\max}(A^\top A)$ and $\lambda_{\max}(A^\top A)$ denotes the largest eigenvalue of $A^\top A$.
We assume that \eqref{alg:eg-adm-sub-x} is relatively easy to solve. Basically, when $A$ is an identity matrix, and $f$ is a function arising from sparse optimization such as the $\ell_1$ norm, $\ell_2$ norm, nuclear norm and so on, then
the subproblem \eqref{alg:eg-adm-sub-x} is usually easy to solve.

The EGADM \eqref{alg:eg-adm} can be written equivalently in a more compact form. By defining
\[z := \bpm y \\ \lambda \epm, \bar{z} := \bpm \bar{y} \\ \bar{\lambda} \epm, F(x, z) = \bpm \nabla_y \LCal(x,y;\lambda) \\ -\nabla_{\lambda} \LCal(x,y;\lambda) \epm, \ZCal := \YCal \times \br^m,\]
we can rewrite \eqref{alg:eg-adm} equivalently as
\be\label{alg:eg-adm-rewrite-z}\left\{\ba{lll}
x^{k+1} & := & \argmin_{x\in\XCal} \LCal_{\gamma}(x,y^k;\lambda^k) + \half\|x-x^k\|_{H}^2 \\
\bz^{k+1} & := & [z^k - \gamma F(x^{k+1},z^k)]_{\ZCal} \\
z^{k+1} & := & [z^k - \gamma F(x^{k+1},\bz^{k+1})]_{\ZCal}.
\ea\right.\ee
That is, in each iteration of EGADM, one minimizes the augmented Lagrangian function plus a proximal term with respect to $x$, and then takes extragradient steps for both $y$ and $\lambda$ for the Lagrangian function.

The idea of extragradient is not new. In fact, the extragradient method as we know it was originally proposed by Korpelevich for variational inequalities and for solving saddle-point problems \cite{Korpelevich-EG-1976,Korpelevich-EG-1983}. Korpelevich proved the convergence of the extragradient method \cite{Korpelevich-EG-1976,Korpelevich-EG-1983}. For recent results on convergence of extragradient type methods, we refer to \cite{Noor-extragradient-2003} and the references therein. The iteration complexity of extragradient method was analyzed by Nemirovski in \cite{Nemirovski-Prox-siopt-2005}. Recently, Monteiro and Svaiter \cite{Monteiro-Svaiter-2010,Monteiro-Svaiter-2011,Monteiro-Svaiter-2010a} studied the iteration complexity results of the hybrid proximal extragradient method proposed by Solodov and Svaiter in \cite{Solodov-Svaiter-HPE-1999} and its variants. More recently, Bonettini and Ruggiero studied a generalized extragradient method for total variation based image restoration problem \cite{Bonettini-Ruggiero-2011}.

\section{Iteration Complexity}\label{sec:complexity}

In this section, we analyze the iteration complexity of EGADM, i.e., \eqref{alg:eg-adm}. We show that under the assumption that the smooth function $g$ has a Lipschitz continuous gradient, EGADM \eqref{alg:eg-adm} finds an $\epsilon$-optimal solution 
to Problem \eqref{prob:min-sum-2} in terms of both objective error and constraint violation within $O(1/\epsilon)$ iterations.

For the linearly constrained convex minimization \eqref{prob:min-sum-2}, the most natural way to define the $\epsilon$-optimal solution is in terms of the objective error and constraint violation as follows (see also \cite{Lin-Ma-Zhang-admm-sublinear-2014,Davis-Yin-2014a}).
\begin{definition}\label{def:epsilon-optimal-primal}
We call $(\hat{x},\hat{y})\in\XCal\times\YCal$ an $\epsilon$-optimal solution to Problem \eqref{prob:min-sum-2} in terms of objective error and constraint violation, if the following holds,
\be\label{eps-optimal-definition-primal} |f(\hat{x})+g(\hat{y})-f(x^*)-g(y^*)| = O(\epsilon), \quad \mbox{ and } \quad \|A\hat{x}+B\hat{y}-b\| = O(\epsilon) \ee
where $(x^*,y^*)\in \XCal^\star\times\YCal^\star$ is any optimal solution to \eqref{prob:min-sum-2}.
\end{definition}

Throughout this paper, we assume that the primal-dual optimal solution sets $\XCal^*\times \YCal^*\times\Lambda^*$ of \eqref{prob:min-sum-2} is non-empty. Note that $(x^*,y^*;\lambda^*)$ is an optimal solution to \eqref{prob:min-sum-2} if and only if
$x^*\in\XCal$, $y^*\in\YCal$, and
\be\label{kkt}\left\{\ba{ll} \langle x-x^*, \partial f(x^*)-A^\top\lambda^* \rangle \geq 0, & \forall x\in \XCal, \\
                             \langle y-y^*, \partial g(y^*)-B^\top\lambda^* \rangle \geq 0, & \forall y\in \YCal, \\
                             Ax^*+By^*=b, \ea\right.\ee
where $\partial f(x)$ and $\partial g(y)$ denote subgradients of $f$ and $g$.

Now we are ready to analyze the iteration complexity of EGADM \eqref{alg:eg-adm-rewrite-z}, or equivalently, \eqref{alg:eg-adm}, for an $\epsilon$-optimal solution in the sense of Definition \ref{def:epsilon-optimal-primal}.
We will prove the following lemma first.
\begin{lemma}\label{lem:3-point-inequality}
The sequence $\{x^{k+1},z^k,\bar{z}^k\}$ generated by \eqref{alg:eg-adm-rewrite-z} satisfies the following inequality:
\be\label{lem:3-point-inequality-conclusion} \ba{ll} & \langle \gamma F(x^{k+1}, \bar{z}^{k+1}), \bar{z}^{k+1}-z^{k+1} \rangle - \half\|z^k-z^{k+1}\|^2 \\ \leq & \gamma^2\|F(x^{k+1},\bar{z}^{k+1})-F(x^{k+1},z^k)\|^2 -\half\|\bar{z}^{k+1}-z^k\|^2-\half\|\bar{z}^{k+1}-z^{k+1}\|^2. \ea\ee
\end{lemma}
\begin{proof}
Note that the optimality conditions of the two subproblems for $z$ in \eqref{alg:eg-adm-rewrite-z} are given by
\be\label{optcond-z-1} \langle z^k - \gamma F(x^{k+1},z^k) - \bz^{k+1}, z-\bz^{k+1} \rangle \leq 0, \quad \forall z \in \ZCal, \ee
and
\be\label{optcond-z-2} \langle z^k - \gamma F(x^{k+1},\bz^{k+1}) - z^{k+1}, z-z^{k+1} \rangle \leq 0, \quad \forall z \in \ZCal.\ee
Letting $z=z^{k+1}$ in \eqref{optcond-z-1} and $z=\bz^{k+1}$ in \eqref{optcond-z-2}, and then summing the two resulting inequalities, we get
\be\label{optcond-z-1+optcond-z-2} \|z^{k+1}-\bz^{k+1}\|^2 \leq \gamma\langle F(x^{k+1},z^k)-F(x^{k+1},\bz^{k+1}), z^{k+1}-\bz^{k+1}\rangle, \ee
which implies
\be\label{lem-part-a} \|z^{k+1}-\bz^{k+1}\| \leq \gamma\|F(x^{k+1},z^k)-F(x^{k+1},\bz^{k+1})\|. \ee

Now we are able to prove \eqref{lem:3-point-inequality-conclusion}. We have,
\be\label{upper-bound-phi-zk+1}\ba{lll} & & \langle\gamma F(x^{k+1},\bz^{k+1}), \bz^{k+1}-z^{k+1}\rangle -\half\|z^k-z^{k+1}\|^2 \\
 & = & \gamma \langle F(x^{k+1},\bz^{k+1}) - F(x^{k+1},z^k), \bz^{k+1} - z^{k+1} \rangle \\
 &   & + \gamma \langle F(x^{k+1},z^k), \bz^{k+1} - z^{k+1} \rangle -\half\|z^k-z^{k+1}\|^2 \\
 & \leq & \gamma \langle F(x^{k+1},\bz^{k+1}) - F(x^{k+1},z^k), \bz^{k+1} - z^{k+1} \rangle \\
 &   & + \langle z^k - \bz^{k+1}, \bz^{k+1} -z^{k+1}\rangle -\half\|z^k-z^{k+1}\|^2 \\
 & = & \gamma \langle F(x^{k+1},\bz^{k+1}) - F(x^{k+1},z^k), \bz^{k+1} - z^{k+1} \rangle \\
 &   & -\half\|z^k\|^2 + \langle z^k,\bz^{k+1} \rangle + \langle \bz^{k+1},z^{k+1}-\bz^{k+1} \rangle - \half\|z^{k+1}\|^2 \\
 & \leq & \gamma \|F(x^{k+1},\bz^{k+1})-F(x^{k+1},z^k)\|\cdot\|\bz^{k+1}-z^{k+1}\| \\
 &   & + \left(-\half\|z^k\|^2 + \langle z^k,\bz^{k+1} \rangle - \half\|\bz^{k+1}\|^2\right) + \left( - \half\|\bz^{k+1}\|^2 + \langle \bz^{k+1},z^{k+1} \rangle - \half\|z^{k+1}\|^2 \right) \\
 & \leq & \gamma^2 \|F(x^{k+1},\bz^{k+1})-F(x^{k+1},z^k)\|^2 - \half\|\bz^{k+1}-z^k\|^2 - \half\|\bz^{k+1}-z^{k+1}\|^2, \\
\ea\ee
where the first inequality is obtained by letting $z=z^{k+1}$ in \eqref{optcond-z-1} and the last inequality follows from \eqref{lem-part-a}. This completes the proof.
\end{proof}

We next prove the following lemma.
\begin{lemma}\label{lem:vi-negative}
Assume that $\nabla g(y)$ is Lipschitz continuous with Lipschitz constant $L_g$, i.e.,
\be\label{lip-nabla-f-y} \|\nabla g(y_1) - \nabla g(y_2)\|\leq L_g\|y_1-y_2\|, \ \forall y_1,y_2\in\YCal. \ee
By letting $\gamma\leq 1/(2\hat{L})$, where $\hat{L} := \left(\max\{2 L_g^2+\lambda_{\max}(B^\top B), 2\lambda_{\max}(B^\top B)\}\right)^{\half}$, the following inequality holds,
\be\label{lem:vi-negative-conclusion} \langle \gamma F(x^{k+1}, \bar{z}^{k+1}), \bar{z}^{k+1}-z^{k+1} \rangle - \half\|z^k-z^{k+1}\|^2\leq 0. \ee
\end{lemma}

\begin{proof}
For any $z_1\in\ZCal$ and $z_2\in\ZCal$, we have,
\[\ba{lll} & & \|F(x^{k+1},z_1) - F(x^{k+1},z_2)\|^2 \\ & = & \left\|\bpm (\nabla g(y_1) - B^\top \lambda_1)-(\nabla g(y_2) - B^\top \lambda_2) \\ (Ax^{k+1}+By_1-b)-(Ax^{k+1}+By_2-b)\epm\right\|^2 \\
 & = & \|(\nabla g(y_1)-\nabla g(y_2)) - B^\top (\lambda_1 - \lambda_2)\|^2 + \|B(y_1-y_2)\|^2 \\
 & \leq & 2 \|\nabla g(y_1)-\nabla g(y_2)\|^2 + 2 \|B^\top (\lambda_1 - \lambda_2)\|^2 + \|B(y_1-y_2)\|^2 \\
 & \leq & 2 L_g^2\|y_1-y_2\|^2 + 2\lambda_{\max}(B^\top B)\|\lambda_1 - \lambda_2\|^2 + \lambda_{\max}(B^\top B)\|y_1-y_2\|^2 \\
 & \leq & \max\{2 L_g^2+\lambda_{\max}(B^\top B), 2\lambda_{\max}(B^\top B)\} \left\|\bpm y_1-y_2 \\ \lambda_1 - \lambda_2 \epm\right\|^2 \\
 & = & \hat{L}^2 \|z_1-z_2\|^2,
\ea\]
where the second inequality is due to \eqref{lip-nabla-f-y} and the last equality is from the definition of $\hat{L}$. Thus, we know that $F(x^{k+1},z)$ is Lipschitz continuous with Lipschitz constant $\hat{L}$. Since $\gamma\leq 1/(2\hat{L})$, we have the following inequality,
\[\ba{lll}
 &      & \gamma^2 \|F(x^{k+1},\bz^{k+1})-F(x^{k+1},z^k)\|^2 - \half\|\bz^{k+1}-z^k\|^2 - \half\|\bz^{k+1}-z^{k+1}\|^2 \\
 & \leq & \gamma^2 \|F(x^{k+1},\bz^{k+1})-F(x^{k+1},z^k)\|^2 - \half\|\bz^{k+1}-z^k\|^2 \\
 & \leq & (\gamma^2 \hat{L}^2 - \half) \|\bz^{k+1}-z^k\|^2 \\
 & \leq & 0,\ea\]
which combining with \eqref{lem:3-point-inequality-conclusion} yields \eqref{lem:vi-negative-conclusion}.
\end{proof}

We further prove the following lemma.
\begin{lemma}\label{lem:main-vi}
Under the same assumptions as in Lemma \ref{lem:vi-negative}, the following holds:
\be\label{lem:main-vi-conclusion}\half\|z-z^{k+1}\|^2-\half\|z-z^k\|^2 \leq \langle\gamma F(x^{k+1},\bz^{k+1}), z-\bz^{k+1}\rangle. \ee
\end{lemma}
\begin{proof}
Adding
\[\langle\gamma F(x^{k+1},\bz^{k+1}), z-z^{k+1}\rangle - \half\|z^{k+1}-z^k\|^2\]
to both sides of \eqref{optcond-z-2}, we get,
\be\label{proof-lem-main-eq-1}\langle z^k-z^{k+1},z-z^{k+1} \rangle - \half\|z^{k+1}-z^k\|^2 \leq \langle\gamma F(x^{k+1},\bz^{k+1}), z-z^{k+1}\rangle - \half\|z^{k+1}-z^k\|^2. \ee
Notice that the left hand side of \eqref{proof-lem-main-eq-1} is equal to $\half\|z-z^{k+1}\|^2-\half\|z-z^k\|^2$. Thus we have,
\be\label{proof-lem-main-eq-2}\ba{ll} & \half\|z-z^{k+1}\|^2-\half\|z-z^k\|^2 \\ \leq & \langle\gamma F(x^{k+1},\bz^{k+1}), z-z^{k+1}\rangle - \half\|z^{k+1}-z^k\|^2 \\ = & \langle\gamma F(x^{k+1},\bz^{k+1}), z-\bz^{k+1}\rangle + \langle\gamma F(x^{k+1},\bz^{k+1}), \bz^{k+1}-z^{k+1}\rangle-\half\|z^{k+1}-z^k\|^2 \\ \leq & \langle\gamma F(x^{k+1},\bz^{k+1}), z-\bz^{k+1}\rangle, \ea\ee
where the last inequality is due to \eqref{lem:vi-negative-conclusion}.
\end{proof}

We now give the $O(1/\epsilon)$ iteration complexity of \eqref{alg:eg-adm} (or equivalently, \eqref{alg:eg-adm-rewrite-z}) for an $\epsilon$-optimal solution to Problem \eqref{prob:min-sum-2}.
\begin{theorem}
Consider Algorithm EGADM~(\ref{alg:eg-adm}), and its sequence of iterates.
For any integer $N>0$, define
\[\tilde{x}^N := \frac{1}{N}\sum_{k=0}^N x^{k+1}, \ \tilde{y}^N := \frac{1}{N}\sum_{k=0}^N \bar{y}^{k+1}, \ \tilde{\lambda}^N := \frac{1}{N}\sum_{k=0}^N \bar{\lambda}^{k+1}.\]
Assume that $\nabla g(y)$ is Lipschitz continuous with Lipschitz constant $L_g$, and we choose $\gamma\leq 1/(2\hat{L})$, where $\hat{L} := \left(\max\{2 L_g^2+\lambda_{\max}(B^\top B), 2\lambda_{\max}(B^\top B)\}\right)^{\half}$. Moreover, we choose $H:=0$ if $A$ is an identity matrix, and $H:=\tau I - \gamma A^\top A$ when $A$ is not identity, where $\tau > \gamma \lambda_{\max}(A^\top A)$. 
For any optimal solution $(x^*,y^*,\lambda^*)\in\XCal^\star\times\YCal^\star\times\Lambda^*$ to \eqref{prob:min-sum-2},
it holds that
\be\label{eq:theorem-inequality}|f(\tilde{x}^N)+g(\tilde{y}^N)-f(x^*)-g(y^*)| = O(1/N), \quad \mbox{ and } \quad \|A\tilde{x}^N+B\tilde{y}^N-b\| = O(1/N). \ee
\end{theorem}

Note that \eqref{eq:theorem-inequality} implies that when $N=O(1/\epsilon)$, $\{\tilde{x}^N,\tilde{y}^N,\tilde{\lambda}^N\}$ is an $\epsilon$-optimal solution to Problem \eqref{prob:min-sum-2} in the sense of Definition \ref{def:epsilon-optimal-primal}; i.e., the iteration complexity of \eqref{alg:eg-adm} (or equivalently, \eqref{alg:eg-adm-rewrite-z}) for an $\epsilon$-optimal solution to Problem \eqref{prob:min-sum-2} is $O(1/\epsilon)$ in terms of both objective error and constraint violation.

\begin{proof}
The optimality conditions of the subproblem for $x$ in \eqref{alg:eg-adm} are given by
\be\label{optcond-x}\langle \theta^{k+1}-A^\top \lambda^k+\gamma A^\top(Ax^{k+1}+By^k-b)+H(x^{k+1}-x^k),x-x^{k+1}\rangle\geq 0, \quad \forall x\in\XCal,\ee
where $\theta^{k+1}\in\partial f(x^{k+1})$ denotes a subgradient of $f$ at point $x^{k+1}$.
By using the updating formula for $\blambda^{k+1}$ in \eqref{alg:eg-adm}, i.e.,
\[\blambda^{k+1} := \lambda^k + \gamma \nabla_{\lambda} \LCal(x^{k+1},y^k;\lambda^k) = \lambda^k - \gamma (Ax^{k+1}+By^k-b),\]
we obtain,
\be\label{optcond-x-rewrite} \langle \theta^{k+1}-A^\top \blambda^{k+1}+H(x^{k+1}-x^k),x-x^{k+1}\rangle\geq 0, \quad \forall x\in\XCal.\ee
Combining \eqref{lem:main-vi-conclusion} and \eqref{optcond-x-rewrite}, we have,
\be\label{the:main-eq-1}\ba{ll}& \left\langle\bpm x-x^{k+1}\\ y-\by^{k+1} \\ \lambda-\blambda^{k+1}\epm, \bpm \theta^{k+1}-A^\top \blambda^{k+1} \\ \nabla g(\by^{k+1}) - B^\top\blambda^{k+1} \\ Ax^{k+1}+B\by^{k+1}-b \epm \right\rangle \\ \geq & \frac{1}{2\gamma}\left(\|z-z^{k+1}\|^2-\|z-z^k\|^2\right)+\langle H(x^k-x^{k+1}), x-x^{k+1}\rangle.\ea\ee

Note that by using the convexity of functions $f$ and $g$ and letting $x=x^*$ and $y=y^*$, \eqref{the:main-eq-1} implies that the following holds for any $\lambda\in \br^m$:
\begin{eqnarray*}
& & f(x^*)+g(y^*)-f(x^{k+1})-g(\bar{y}^{k+1})+\left(\begin{array}{c} x^* - x^{k+1} \\ y^* - \bar{y}^{k+1} \\ \lambda-\bar{\lambda}^{k+1}\end{array} \right)^\top
\left(\begin{array}{c} -A^\top\bar{\lambda}^{k+1} \\ -B^\top\bar{\lambda}^{k+1} \\  A x^{k+1} + B\bar{y}^{k+1} - b \end{array} \right)\\
& & + \left( \frac{1}{2}\left\| x^*-x^k \right\|_H^2 + \frac{1}{2\gamma}\left\| y^*-y^k \right\|^2 + \frac{1}{2\gamma}\left\| \lambda-\lambda^k \right\|^2 \right) \\
& & - \left( \frac{1}{2}\left\| x^*-x^{k+1} \right\|_H^2 + \frac{1}{2\gamma}\left\| y^*-y^{k+1} \right\|^2 + \frac{1}{2\gamma}\left\| \lambda-\lambda^{k+1} \right\|^2 \right)  \\
& \geq & 0,
\end{eqnarray*}
which further yields that
\begin{eqnarray}\label{primal-opt-main-inequality}
\nonumber &   & f(x^*)+g(y^*)-f(\tilde{x}^N)-g(\tilde{y}^N) + \lambda^\top \left( A\tilde{x}^N + B \tilde{y}^N-b \right) \\ \nonumber
& = & f(x^*)+g(y^*)-f(\tilde{x}^N)-g(\tilde{y}^N)+\left(\begin{array}{c} x^* - \tilde{x}^N \\ y^* - \tilde{y}^N \\ \nonumber \lambda-\tilde{\lambda}^N\end{array} \right)^{\top}
\left(\begin{array}{c} -A^\top\tilde{\lambda}^N \\ -B^\top\tilde{\lambda}^N \\ A\tilde{x}^N + B \tilde{y}^N-b \end{array} \right)\\
\nonumber & \geq & \frac{1}{N+1}\sum\limits_{k=0}^{N}\left[f(x^*)+g(y^*)-f(x^{k+1})-g(\bar{y}^{k+1})+\left(\begin{array}{c} x^* - x^{k+1} \\ y^* - \bar{y}^{k+1} \\ \lambda-\bar{\lambda}^{k+1}\end{array} \right)^\top
\left(\begin{array}{c} -A^\top\bar{\lambda}^{k+1} \\ -B^\top\bar{\lambda}^{k+1} \\  A x^{k+1} + B\bar{y}^{k+1} - b \end{array} \right)\right] \\
\nonumber & \geq &\frac{1}{N+1}\sum\limits_{k=0}^N\left[ \left( \frac{1}{2}\left\| x^*-x^{k+1} \right\|_H^2 + \frac{1}{2\gamma}\left\| y^*-y^{k+1} \right\|^2 + \frac{1}{2\gamma}\left\| \lambda-\lambda^{k+1} \right\|^2 \right) \right. \\
\nonumber & & \left. -\left( \frac{1}{2}\left\| x^*-x^k \right\|_H^2 + \frac{1}{2\gamma}\left\| y^*-y^k \right\|^2 + \frac{1}{2\gamma}\left\| \lambda-\lambda^k \right\|^2 \right)\right]\\
& \geq & -\frac{1}{2\gamma(N+1)}\left\|\lambda - \lambda^0 \right\|^2 -\frac{1}{2\gamma(N+1)}\left\| y^* - y^0 \right\|^2 - \frac{1}{2(N+1)} \left\| x^* - x^0 \right\|_H^2,
\end{eqnarray}
where the first inequality is due to the convexity of $f$ and $g$.
Note that the optimality condition \eqref{kkt} and convexity of $f$ and $g$ imply the following inequality
\be\label{primal-opt-main-inequality-1}0 \geq f(x^*)+g(y^*)-f(\tilde{x}^N)-g(\tilde{y}^N) + \langle\lambda^*, A\tilde{x}^N + B \tilde{y}^N-b \rangle.\ee
Now, define $\rho:=\|\lambda^*\|+1$. By using Cauchy-Schwarz inequality in \eqref{primal-opt-main-inequality-1}, we obtain
\begin{eqnarray}\label{ergodic-N-inequality-1}
0  \leq f(\tilde{x}^N)+g(\tilde{y}^N) - f(x^*)-g(y^*)+ \rho\left\| A\tilde{x}^N + B \tilde{y}^N-b \right\|.
\end{eqnarray}
By setting $\lambda = -\rho(A\tilde{x}^N + B \tilde{y}^N-b)/\|A\tilde{x}^N + B \tilde{y}^N-b\|$ in \eqref{primal-opt-main-inequality}, and noting that $\|\lambda\|=\rho$, we obtain
\begin{eqnarray}\label{ergodic-N-inequality-2}
& f(\tilde{x}^N)+g(\tilde{y}^N) - f(x^*)-g(y^*)+ \rho\left\| A\tilde{x}^N + B \tilde{y}^N-b \right\| \nonumber \\
 \leq & \frac{\rho^2 + \|\lambda^0\|^2}{\gamma (N+1)} + \frac{1}{2\gamma(N+1)}\left\| y^* - y^0 \right\|^2 + \frac{1}{2(N+1)} \left\| x^* - x^0 \right\|_H^2.
\end{eqnarray}
We now define the function
\[v(\xi) = \min \{ f(x)+g(y) | Ax + By - b = \xi, x\in \XCal, y\in\YCal \}.\]
It is easy to verify that $v$ is convex, $v(0)=f(x^*)+g(y^*)$, and $\lambda^* \in \partial v(0)$.
Therefore, from the convexity of $v$, it holds that
\be\label{v-convex} v(\xi) \ge v(0) + \langle \lambda^*, \xi \rangle \ge f(x^*)+g(y^*) - \| \lambda^*\| \|\xi\|.\ee
Let $\bar{\xi} = A \tilde{x}^N + B\tilde{y}^N - b$, we have $f(\tilde{x}^N)+g(\tilde{y}^N) \ge v(\bar{\xi})$.
Therefore, combining \eqref{ergodic-N-inequality-1}, \eqref{ergodic-N-inequality-2} and \eqref{v-convex}, we get
\begin{eqnarray*}
- \| \lambda^*\| \| \bar{\xi} \| & \leq & f(\tilde{x}^N)+g(\tilde{y}^N) - f(x^*)- g(y^*) \\
& \leq & \frac{\rho^2 + \|\lambda^0\|^2}{\gamma (N+1)} + \frac{1}{2\gamma(N+1)}\left\| y^* - y^0 \right\|^2 + \frac{1}{2(N+1)} \left\| x^* - x^0 \right\|_H^2 - \rho\| \bar{\xi} \|,
\end{eqnarray*}
which by defining
\[C := \frac{\rho^2+\|\lambda^0\|^2}{\gamma}+\frac{\|y^*-y^0\|^2}{2\gamma}+\half\|x^*-x^0\|_H^2\]
yields,
\begin{eqnarray}\label{infeasibility}
\| A \tilde{x}^N + B\tilde{y}^N - b \| = \|\bar{\xi}\| \leq \frac{C}{N+1},
\end{eqnarray}
and
\begin{eqnarray}\label{obj-diff}
-\frac{\|\lambda^*\|C}{N+1}\leq f(\tilde{x}^N)+g(\tilde{y}^N) - f(x^*)- g(y^*) \leq \frac{C}{N+1}.
\end{eqnarray}
Combining \eqref{infeasibility} and \eqref{obj-diff} completes the proof.
\end{proof}

\begin{remark}
We note that there is another line of research on studying dual smoothing methods for solving \eqref{prob:min-sum-2}. This approach is studied in \cite{Necoara-Suykends-2008,Dinh-Necoara-Diehl-2014}. The dual smoothing method in \cite{Necoara-Suykends-2008} smoothes the Lagrangian dual function using Nesterov's smoothing technique in \cite{Nesterov-2005}, and then applies the accelerated gradient method \cite{Nesterov-2005} to solve the smoothed dual problem. The accelerated gradient method can return an $\epsilon$-optimal solution to \eqref{prob:min-sum-2} in $O(1/\epsilon)$ iterations. It should be noted that the smoothing technique used in \cite{Necoara-Suykends-2008} requires the feasible sets $\XCal$ and $\YCal$ to be bounded, while this assumption is not needed in our EGADM method. On the other hand, the dual smoothing method in \cite{Necoara-Suykends-2008} does not assume any smoothness assumptions on $f$ and $g$, while our EGADM requires $g$ to be differentiable. In \cite{Dinh-Necoara-Diehl-2014}, the authors propose to smooth the Lagrangian dual function using self-concordant barrier functions, and then apply a path-following gradient method to minimize the smoothed dual function. It is noted that this method also requires the feasible sets $\XCal$ and $\YCal$ to be bounded, and the $O(1/\epsilon)$ complexity result is only measured by the dual objective value error.
\end{remark}


\section{Fused Logistic Regression} \label{sec:fused-logistic}

In this section we show how to apply Algorithm \eqref{alg:eg-adm} to solve the fused logistic regression problem, which is a convex problem. To introduce our fused logistic regression model, we need to introduce fused lasso problem and logistic regression first.
The sparse linear regression problem, known as Lasso \cite{Tibshirani-96}, was introduced to find sparse regression coefficients so that the resulting model is more interpretable. The original Lasso model solves the following problem:
\be\label{lasso} \min \ \half\|Ax-b\|^2, \ \st \ \|x\|_1 \leq s, \ee
where $A=[a_1,\ldots,a_m]^\top \in\br^{m\times n}$ gives the predictor variables, $b=[b_1,\ldots,b_m]^\top\in\br^m$ gives the responses and the constraint $\|x\|_1\leq s$ is imposed to promote the sparsity of the regression coefficients $x$. The Lasso solution $x$ gives more interpretability to the regression model since the sparse solution $x$ ensures that only a few features contribute to the prediction.

Fused lasso was introduced by Tibshirani \etal in \cite{Tibshirani-fused-lasso-2005} to model the situation in which there is certain natural ordering in the features. Fused lasso adds a term to impose the sparsity of $x$ in the gradient space to model natural ordering in the features. The fused lasso problem can be formulated as
\be\label{fused-lasso} \min \ \half\|Ax-b\|^2 + \alpha\|x\|_1 + \beta\sum_{j=2}^n |x_j-x_{j-1}|. \ee
Because \eqref{fused-lasso} can be transformed equivalently to a quadratic programming problem, Tibshirani \etal proposed to solve \eqref{fused-lasso} using a two-phase active set algorithm SQOPT of Gill \etal \cite{Gill-SQOPT-1997}. However, transforming \eqref{fused-lasso} to a quadratic programming problem will increase the size of the problem significantly, thus SQOPT can only solve \eqref{fused-lasso} with small or medium sizes. Ye and Xie \cite{Ye-Xie-fused-lasso} proposed to solve \eqref{fused-lasso} using split Bregman algorithm, which can be shown to be equivalent to an alternating direction method of multipliers.
Note that \eqref{fused-lasso} can be rewritten equivalently as
\be\label{fused-lasso-rewrite} \ba{ll} \min & \half\|Ax-b\|^2 + \alpha \|w\|_1 + \beta \|y\|_1 \\
                                        \st & w = x \\
                                            & y = Lx, \ea \ee
where $L$ is an $(n-1)\times n$ dimensional matrix with all ones on the diagonal and negative ones on the super-diagonal and zeros elsewhere. The ADMM can be applied to solve \eqref{fused-lasso-rewrite} with $x$ being one block variable and $(w,y)$ being the other block (see \cite{Ye-Xie-fused-lasso} for more details).

As in \cite{Liu-logistic-regression-2009}, the sparse logistic regression problem \eqref{logistic-reg} can be formulated as
\be\label{logistic-reg-L1ball} \min_{x,c} \ \ell(x,c), \ \st \  \|x\|_1 \leq s. \ee

It is now very meaningful to consider the fused logistic regression problem when there is certain natural ordering in the features. This leads to the following optimization problem:
\be\label{fused-logistic-reg} \min_{x\in\br^n,c\in\br} \ \ell(x,c) + \alpha \|x\|_1 + \beta\sum_{j=2}^n |x_j-x_{j-1}|. \ee
Problem \eqref{fused-logistic-reg} can be rewritten equivalently as
\be\label{fused-logistic-reg-rewrite} \ba{ll}\min_{x\in\br^n,w\in\br^{n-1},y\in\br^n,c\in\br} & \alpha \|x\|_1 + \beta\|w\|_1 + \ell(y,c)  \\ \st & x = y \\ & w = Ly. \ea\ee
If we apply the ADMM to solve \eqref{fused-logistic-reg-rewrite}, we will end up with the following iterates:
\be\label{alg:adm-fused-log-reg} \left\{ \ba{lll} (x^{k+1},w^{k+1}) & := & \argmin_{x,w} \ \LCal_{\gamma}(x,w,y^k,c^k;\lambda_1^k,\lambda_2^k) \\
                                                  (y^{k+1},c^{k+1}) & := & \argmin_{y,c} \ \LCal_{\gamma}(x^{k+1},w^{k+1},y,c;\lambda_1^k,\lambda_2^k) \\
                                                  \lambda_1^{k+1} & := & \lambda_1^k - \gamma (x^{k+1}-y^{k+1}) \\
                                                  \lambda_2^{k+1} & := & \lambda_2^k - \gamma (w^{k+1}-Ly^{k+1}), \ea\right.  \ee
where the augmented Lagrangian function $\LCal_{\gamma}(x,w,y;\lambda_1,\lambda_2)$ is defined as
\[\ba{ll} \LCal_{\gamma}(x,w,y,c;\lambda_1,\lambda_2) :=  & \frac{1}{m}\sum_{i=1}^m \log(1+\exp(-b_i(a_i^\top y+c))) \\ & +\alpha\|x\|_1+\beta\|w\|_1-\langle\lambda_1,x-y\rangle-\langle\lambda_2,w-Ly\rangle+\frac{\gamma}{2}\|x-y\|^2+\frac{\gamma}{2}\|w-Ly\|^2.\ea\]
However, note that although the subproblem for $(x,w)$ is still easy, the subproblem for $y$ is no longer easy because of the logistic loss function $\ell(y,c)$.  But, since $\ell(y,c)$ is differentiable with respect to $(y,c)$, we can apply our EGADM to solve \eqref{fused-logistic-reg-rewrite}.
Noting that the subproblem for $(x,w)$ corresponds to two $\ell_1$ shrinkage operations. Moreover, the gradients of $\ell(y,c)$ with respect to $y$ and $c$ are easily obtainable as
\be\label{gradient-log-loss-y-c}\nabla_y\ell(y,c) = -\frac{1}{m}\hat{A}^\top (1-d), \ \nabla_c\ell(y,c)=-\frac{1}{m}b^\top(1-d), \ d = 1./(1+\exp(-\hat{A}y-cb)),\ee where $\hat{A}=[b_1a_1,b_2a_2,\ldots,b_ma_m]^\top$ and $1./\alpha$ denotes the component-wise division.

We are now ready to describe the details of EGADM for solving \eqref{fused-logistic-reg-rewrite} in
Algorithm \ref{alg:eg-adm-fused-log-reg-formal}, in which the $\ell_1$ shrinkage operation is defined as:
\[ \Shrink(z,\tau) := \sign(z) \circ \max\{|z|-\tau,0\}. \]

\begin{algorithm}\caption{Extragradient-based ADM for the Fused Logistic Regression}\label{alg:eg-adm-fused-log-reg-formal}
\begin{algorithmic}
\State Initialization: $\hat{A}=[b_1a_1,b_2a_2,\ldots,b_ma_m]^\top$
\For{$k=0,1,\ldots$}
\State $x^{k+1} := \Shrink(y^k+\lambda_1^k/\gamma,\alpha/\gamma)$
\State $w^{k+1} := \Shrink(Ly^k+\lambda_2^k/\gamma,\beta/\gamma)$
\State $d^k:=1./(1+\exp(-\hat{A}y^k-b\circ c^k))$, $\nabla_y\ell(y^k,c^k):=-\frac{1}{m}\hat{A}^\top(1-d^k)$, $\nabla_c\ell(y^k,c^k):=-\frac{1}{m}b^\top(1-d^k)$
\State $\bar{y}^{k+1} := y^k - \gamma (\nabla_{y}\ell(y^k,c^k)+\lambda_1^k+L^\top\lambda_2^k)$
\State $\bar{c}^{k+1} := c^k - \gamma \nabla_{c}\ell(y^k,c^k)$
\State $\bar{\lambda}_1^{k+1} := \lambda_1^k - \gamma (x^{k+1}-y^k)$, $\bar{\lambda}_2^{k+1} := \lambda_2^k - \gamma (w^{k+1}-Ly^{k+1})$
\State $\bar{d}^{k+1}:=1./(1+\exp(-\hat{A}\bar{y}^{k+1}-b\circ \bar{c}^{k+1}))$
\State $\nabla_y\ell(\bar{y}^{k+1},\bar{c}^{k+1}):=-\frac{1}{m}\hat{A}^\top(1-\bar{d}^{k+1})$, $\nabla_c\ell(\bar{y}^{k+1},\bar{c}^{k+1}):=-\frac{1}{m}b^\top(1-\bar{d}^{k+1})$
\State $y^{k+1} := y^k - \gamma (\nabla_{y}\ell(\bar{y}^{k+1},\bar{c}^{k+1})+\bar{\lambda}_1^{k+1}+L^\top\bar{\lambda}_2^{k+1})$
\State $c^{k+1} := c^k - \gamma \nabla_{c}\ell(\bar{y}^{k+1},\bar{c}^{k+1})$
\State $\lambda_1^{k+1} := \lambda_1^k - \gamma (x^{k+1}-\bar{y}^{k+1})$, $\lambda_2^{k+1} := \lambda_2^k - \gamma (w^{k+1}-L\bar{y}^{k+1})$
\EndFor
\end{algorithmic}
\end{algorithm}

\section{Numerical Experiments}\label{sec:numerical}
In this section, we test the performance of our EGADM for solving the fused logistic regression problem \eqref{fused-logistic-reg} and lasso problem.
Our codes were written in MATLAB. All numerical
experiments were run in MATLAB 7.12.0 on a laptop with Intel Core
I5 2.5 GHz CPU and 4GB of RAM.

\subsection{Numerical Results for Fused Logistic Regression}\label{sec:numerical-fused-logistic}

In this subsection, we report the results of our EGADM (Algorithm \ref{alg:eg-adm-fused-log-reg-formal}) for solving the fused logistic regression problem \eqref{fused-logistic-reg}.

First, we used a very simple example to show that when the features have natural ordering, the fused logistic regression model \eqref{fused-logistic-reg} is much preferable than the sparse logistic regression model \eqref{logistic-reg-L1ball}. This simple example was created in the following manner. We created the regression coefficient $\hat{x}\in\br^n$ for $n=1000$ as
\be\label{random-predictor-simple} \hat{x}_j = \left\{\ba{ll} r_1, & j = 1,2,\ldots,100 \\
                                                              r_2, & j = 201,202,\ldots,300 \\
                                                              r_3, & j = 401,402,\ldots,500 \\
                                                              r_4, & j = 601,602,\ldots,700 \\
                                                              0,   & \mbox{ else}, \ea\right.\ee
where scalers $r_1,r_2,r_3,r_4$ were created randomly uniform in $(0,20)$. An example plot of $\hat{x}$ is shown in the left part of Figure \ref{fig:random-original}. The entries of matrix $A \in\br^{m\times n}$ with $m=500$ and $n=1000$ were drawn from standard normal distribution $\mathcal{N}(0,1)$. Vector $b\in\br^m$ was then created as the signs of $A\hat{x}+ce$, where $c$ is a random number in $(0,1)$ and $e$ is the $m$-dimensional vector of all ones. We then applied our extragradient-based ADM (Algorithm \ref{alg:eg-adm-fused-log-reg-formal}) for solving the fused logistic regression problem \eqref{fused-logistic-reg} and compared the result with the sparse logistic regression problem \eqref{logistic-reg-L1ball}. The code for solving \eqref{logistic-reg-L1ball}, which is called Lassplore and proposed by Liu \etal in \cite{Liu-logistic-regression-2009}, was downloaded from http://www.public.asu.edu/$\sim$jye02/Software/lassplore/. Default settings of Lassplore were used.  We chose $\alpha=5\times 10^{-4}$ and $\beta=5\times 10^{-2}$ in \eqref{fused-logistic-reg}. The regression result by EGADM (Algorithm \ref{alg:eg-adm-fused-log-reg-formal}) is plotted in Figure \ref{fig:fused-random} (a). We tested different choices of $s=1,5,10$ in \eqref{logistic-reg-L1ball} and the results are plotted in Figure \ref{fig:fused-random} (b), (c) and (d), respectively. From Figure \ref{fig:fused-random} we see that, the fused logistic regression model \eqref{fused-logistic-reg} can preserve the natural ordering very well. The sparse logistic regression model \eqref{logistic-reg-L1ball} gives very sparse solution when $s$ is small, and gives less sparse solution when $s$ is large, but none of the choices of $s=1,5,10$ gives a solution that preserves the natural ordering.

\begin{figure}
\centering \subfigure{\includegraphics[scale=0.5]{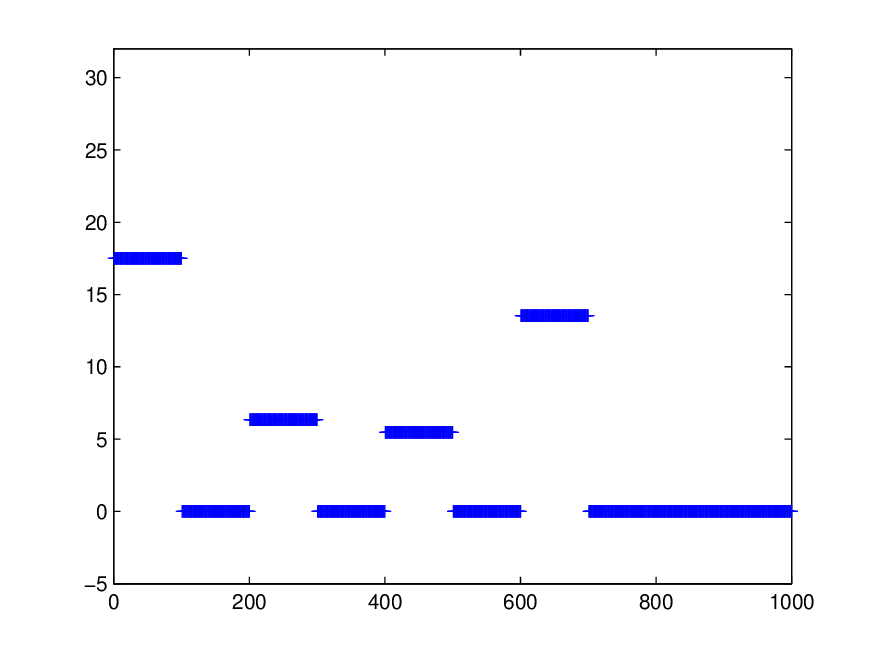}}
\centering \subfigure{\includegraphics[scale=0.5]{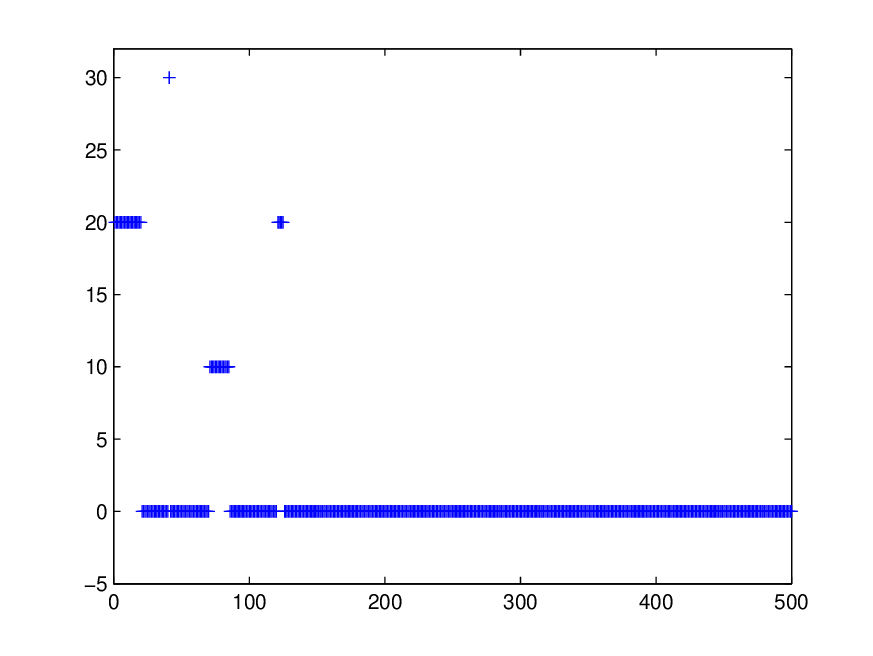}}
\caption{Left: The regression coefficient given in \eqref{random-predictor-simple}; Right: The regression coefficient given in \eqref{random-predictor}.}
\label{fig:random-original}
\end{figure}

To further show the capability of our EGADM for solving the fused logistic regression model \eqref{fused-logistic-reg}, especially for large-scale problems, we conducted the following tests. First, we created the regression coefficient $\hat{x}\in\br^n$ for $n\geq 100$ as
\be\label{random-predictor}\hat{x}_j = \left\{\ba{ll} 20, & j=1,2,\ldots,20,\\
                       30, & j=41, \\
                       10, & j=71,\ldots,85,\\
                       20, & j=121,\ldots,125,\\
                       0 , & \mbox{ else}.\ea\right.\ee
Note that a similar test example was used in \cite{Ye-Xie-fused-lasso} for the fused lasso problem.
An example plot of $\hat{x}$ of size $n=500$ is shown in the right part of Figure \ref{fig:random-original}. We then created matrix $A$ and vector $b$ in the same way mentioned above. We applied our EGADM to solve the fused logistic regression model \eqref{fused-logistic-reg} with the above mentioned inputs $A$ and $b$. We report the iteration number, CPU time, sparsity of $x$ (denoted by $\|x\|_0$) and sparsity of the fused term $Lx$ (denoted by $\|Lx\|_0$) in Table \ref{tab:fused-logistic-reg}. From Table \ref{tab:fused-logistic-reg} we see that our EGADM can solve the fused logistic regression problem \eqref{fused-logistic-reg} efficiently. It solved instances with size up to $m=2000$, $n=20000$ in just a few seconds.

\begin{figure}
\centering \subfigure{\includegraphics[scale=0.5]{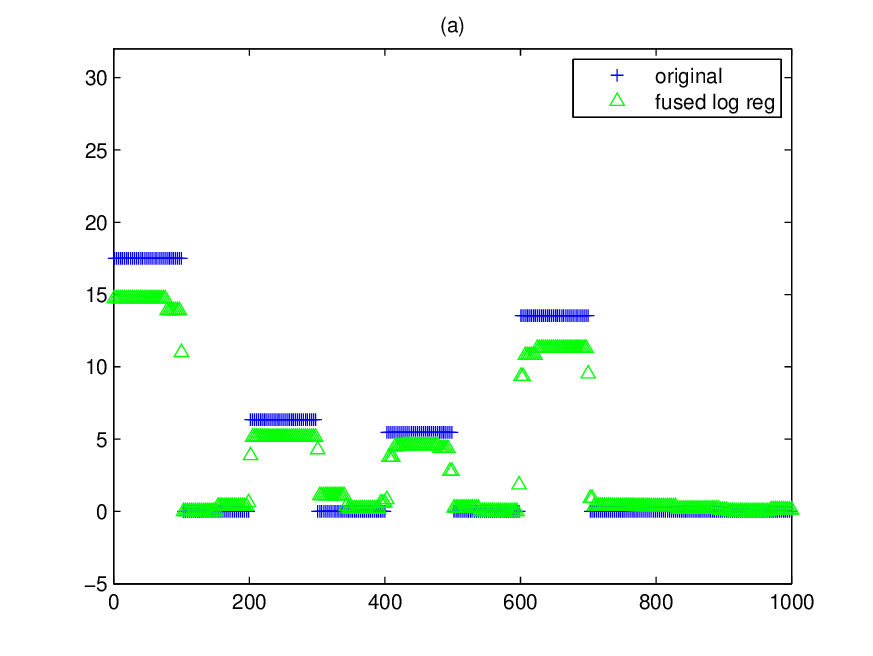}}
\centering \subfigure{\includegraphics[scale=0.5]{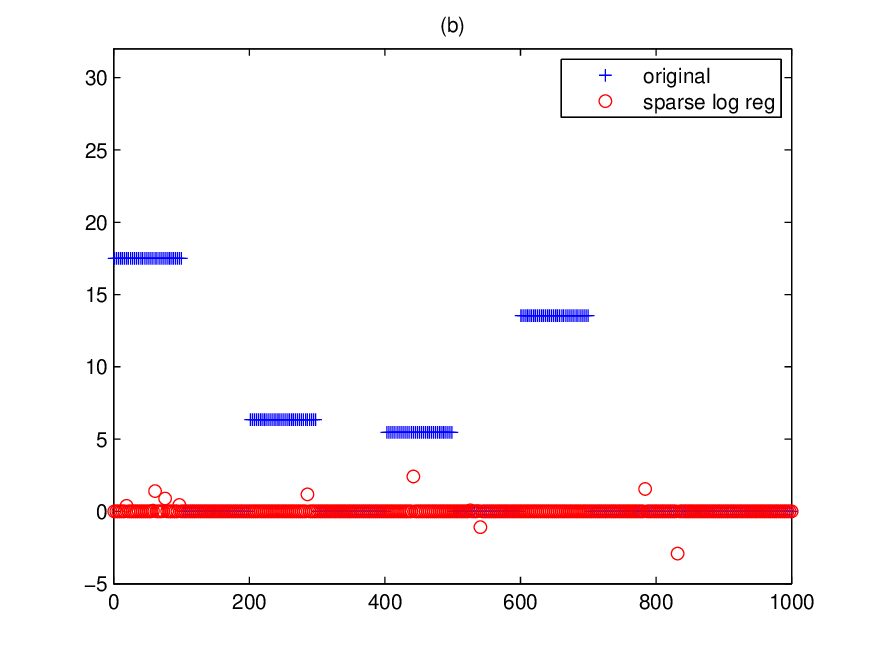}}
\centering \subfigure{\includegraphics[scale=0.5]{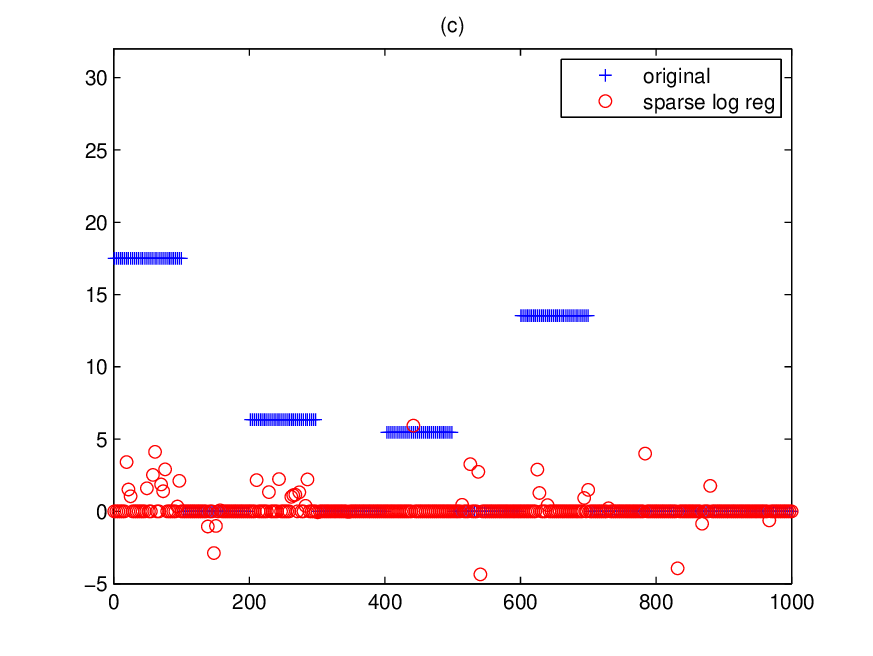}}
\centering \subfigure{\includegraphics[scale=0.5]{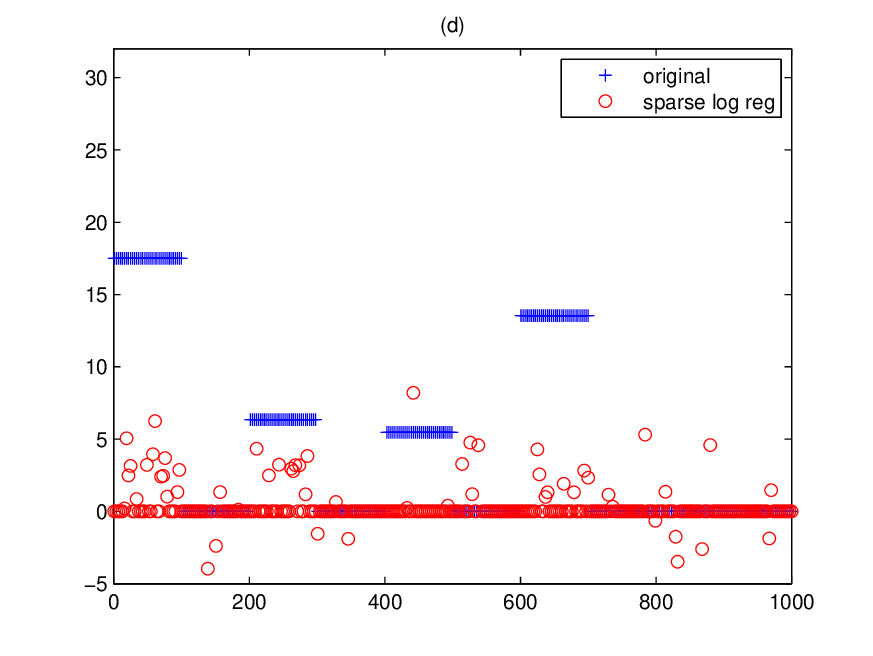}}
\caption{(a): The regression result by the fused logistic regression model \eqref{fused-logistic-reg}; (b), (c), (d): The regression result by the sparse logistic regression model \eqref{logistic-reg-L1ball} with $s=1,5,10$, respectively.}
\label{fig:fused-random}
\end{figure}

\begin{table}[ht]
\begin{center}\caption{{Numerical Results for Fused Logistic Regression}}\label{tab:fused-logistic-reg}
\begin{tabular}{| c c  | c c c c |}\hline
$m$  & $n$    &  iter   &   cpu    & $\|x\|_0$   & $\|Lx\|_0$   \\ \hline
100  & 500    &  104    &   0.0    & 37          & 40           \\\hline
100  & 1000   &  112    &   0.1    & 41          & 46           \\\hline
100  & 2000   &  105    &   0.2    & 71          & 88            \\\hline
1000 & 2000   &  69     &   0.6    & 25          & 21            \\\hline
1000 & 5000   &  79     &   1.6    & 25          & 26             \\\hline
1000 & 10000  &  53     &   2.0    & 25          & 25             \\\hline
2000 & 5000   &  94     &   3.5    & 40          & 36              \\\hline
2000 & 10000  &  238    &   17.1   & 40          & 6                \\\hline
2000 & 20000  &  84     &   11.5   & 40          & 42               \\\hline
\end{tabular}
\end{center}
\end{table}

\subsection{Numerical Results for Lasso Problem}\label{sec:numerical-lasso}

To understand how our new method compares to other well established methods when all are applicable, in this subsection, we experiment our EGADM in comparison with ISTA \cite{Beck-Teboulle-2009} and ADMM for solving the following unconstrained version of the lasso problem \eqref{lasso}:
\be\label{prob:lasso}
\min \ \tau\|x\|_1 + \half\|Ax-b\|^2,
\ee
where $\tau>0$ is a given weighting parameter. This problem can be naturally solved by ISTA \cite{Beck-Teboulle-2009}. A typical iteration of ISTA for solving \eqref{prob:lasso} can be described as:
\be\label{ista-lasso} x^{k+1} := \argmin_x \ \gamma\tau\|x\|_1 + \half\|x-(x^k-\gamma A^\top(Ax^k-b))\|^2, \ee
where $\gamma > 0$ is the step size of the gradient step. By applying a variable-splitting technique, \eqref{prob:lasso} can be equivalently written as
\be\label{prob:lasso-x-y} \min \ \tau\|x\|_1 + \half\|Ay-b\|^2, \ \st, \ x-y =0, \ee
which can be solved by both EGADM and ADMM. A typical iteration of EGADM for solving \eqref{prob:lasso-x-y} is
\be\label{alg:egadm-lasso}\left\{\ba{lll}
x^{k+1} & := & \argmin_x \ \tau\|x\|_1 - \langle \lambda^k, x-y^k \rangle + \frac{\gamma}{2}\|x-y^k\|^2 \\
\bar{y}^{k+1} & := & y^k - \gamma(A^\top(Ay^k-b)+\lambda^k) \\
\bar{\lambda}^{k+1} & := & \lambda^k - \gamma(x^{k+1}-y^k) \\
y^{k+1} & := & y^k - \gamma(A^\top(A\bar{y}^{k+1}-b)+\bar{\lambda}^{k+1}) \\
\lambda^{k+1} & := & \lambda^k - \gamma(x^{k+1}-\bar{y}^{k+1}),
\ea\right.\ee
and a typical iteration of ADMM for solving \eqref{prob:lasso-x-y} is
\be\label{alg:admm-lasso}\left\{\ba{lll}
x^{k+1} & := & \argmin_x \ \tau\|x\|_1 - \langle \lambda^k, x-y^k \rangle + \frac{\gamma}{2}\|x-y^k\|^2 \\
y^{k+1} & := & \argmin_y \ \half\|Ay-b\|^2 - \langle \lambda^k, x^{k+1}-y \rangle + \frac{\gamma}{2}\|x^{k+1}-y\|^2 \\
\lambda^{k+1} & := & \lambda^k - \gamma(x^{k+1}-y^{k+1}).
\ea\right.\ee
The main computational efforts in these three algorithms are as follows. In each iteration of ISTA, an $\ell_1$ shrinkage operation and two matrix vector multiplications are needed; in each iteration of EGADM, an $\ell_1$ shrinkage operation and four matrix vector multiplications are needed; while in each iteration of ADMM, an $\ell_1$ shrinkage operation and solving one linear system are required. Note that the lasso problem is suitable for ADMM because the $y$-subproblem is still easy to solve. We use this problem as an example to illustrate the relation among ISTA, EGADM and ADMM. In fact, even if the $y$-subproblem is not easy to solve, an inexact version of ADMM can still be applied. For instance, instead of solving the $y$-subproblem in \eqref{alg:admm-lasso} directly by solving a linear system, we can run gradient descent method for $M$ iterations to get an approximate solution to it. Therefore, in our experiments for solving the lasso problem \eqref{prob:lasso}, we implemented the following comparison to compare the performance of ISTA, EGADM and ADMM. In our numerical experiments, we always set $\tau=0.1$ for simplicity.
The instances for different ($m,n$) were created randomly in the following way. We first generated matrix $A\in\mathbb{R}^{m\times n}$ randomly according to normal distribution. We then normalized $A$ such that the largest singular value of $A$ is $1$. This normalization was implemented so that the step size $\gamma$ can be selected easily. A sparse $x$ was then created such that the number of nonzero components was equal to $n/10$, and their positions were selected uniformly randomly. We then set $b=Ax$. We first ran ISTA with $\gamma=1$ for $100$ iterations, and recorded the resulting objective function value of \eqref{prob:lasso} (denoted by $f_I$). We then ran EGADM and ADMM with different $\gamma$ until the objective function value was smaller than $f_I$, or the maximum number of iteration (set as $1000$) was achieved. Moreover, we also ran the inexact version ADMM described above for different $M$ to see the comparison result. The results are reported in Tables \ref{tab:lasso-1} and \ref{tab:lasso-2}. In particular, ADMM-5 and ADMM-10 indicate that we respectively ran the gradient descent method for $5$ and $10$ iterations to get an approximate solution to the $y$-subproblem in \eqref{alg:admm-lasso}; $mvm$ denotes the total number of matrix vector multiplications. In Tables \ref{tab:lasso-1} and \ref{tab:lasso-2}, we reported the results for different $\gamma$. Note that we used the same $\gamma$ for ISTA, EGADM and ADMM. Moreover, the same $\gamma$ was also used as the step size for the gradient descent method for solving the $y$-subproblems in ADMM-5 and ADMM-10.

{\small
\begin{table}[ht]
\begin{center}\caption{{Numerical Results for Lasso problem}}\label{tab:lasso-1}
\begin{tabular}{| l | ll | ll | lll | lll | lll |}\hline
         & \multicolumn{2}{|c|}{ISTA} & \multicolumn{2}{|c|}{ADMM} & \multicolumn{3}{|c|}{ADMM-5} & \multicolumn{3}{|c|}{ADMM-10} & \multicolumn{3}{|c|}{EGADM} \\\hline
$(m,n)$ & iter & cpu & iter & cpu & iter & mvm & cpu & iter & mvm & cpu & iter & mvm & cpu \\\hline
\multicolumn{14}{|c|}{$\gamma=1.0$} \\\hline
 (100,1000) & 100 &    0.1 & 101 &    5.2 & 102 & 510 &    0.1 & 1000 & 10000 &    0.7 & 102 & 408 &    0.0 \\\hline
 (100,2000) & 100 &   0.3 & 101 &   30.7 & 103 & 515 &    0.1 & 604 & 6040 &    0.9 & 102 & 408 &    0.1 \\\hline
 (100,5000) & 100  &    0.8 & 102 &  435.9 & 105 & 525 &    0.6 & 220 & 2200 &    2.2 & 102 & 408 &    0.2 \\\hline
 (100,8000) & 100  &    1.3 & 102 & 1642.8 & 106 & 530 &    3.7 & 321 & 3210 &    5.9 & 102 & 408 &    0.5 \\\hline
 (1000,100) & 100 &    0.1 & 53 &    0.0 & 1000 & 5000 &    0.3 & 1000 & 10000 &    0.6 & 108 & 432 &    0.0 \\\hline
 (1000,200) & 100  &    0.3 & 88 &    0.1 & 1000 & 5000 &    0.5 & 1000 & 10000 &    1.0 & 165 & 660 &    0.0 \\\hline
 (2000,200) & 100  &    0.6 & 55 &    0.1 & 1000 & 5000 &    2.5 & 1000 & 10000 &    4.3 & 99 & 396 &    0.1 \\\hline
 (5000,100) & 100  &    0.7 & 35 &    0.0 & 1000 & 5000 &    4.6 & 1000 & 10000 &    8.2 & 40 & 160 &    0.1 \\\hline
 (5000,200) & 100 &    1.5 & 43 &    0.1 & 1000 & 5000 &   10.4 & 1000 & 10000 &   19.6 & 66 & 264 &    0.3 \\\hline
 (8000,100) & 100  &    1.3 & 45 &    0.1 & 1000 & 5000 &    8.5 & 1000 & 10000 &   16.2 & 46 & 184 &    0.2 \\\hline
 (8000,200) & 100  &    2.3 & 46 &    0.2 & 1000 & 5000 &   16.4 & 1000 & 10000 &   31.2 & 62 & 248 &    0.5 \\\hline
\multicolumn{14}{|c|}{$\gamma=0.8$} \\\hline
 (100,1000) & 100  &    0.1 & 81 &    3.9 & 82 & 410 &    0.0 & 81 & 810 &    0.1 & 127 & 508 &    0.0 \\\hline
 (100,2000) & 100  &    0.3 & 81 &   23.1 & 82 & 410 &    0.1 & 81 & 810 &    0.1 & 127 & 508 &    0.1 \\\hline
 (100,5000) & 100  &    0.8 & 81 &  320.3 & 82 & 410 &    0.5 & 82 & 820 &    0.7 & 127 & 508 &    0.3 \\\hline
 (100,8000) & 100  &    1.2 & 81 & 1295.8 & 82 & 410 &    0.8 & 82 & 820 &    1.4 & 127 & 508 &    0.6 \\\hline
 (1000,100) & 100  &    0.1 & 42 &    0.0 & 44 & 220 &    0.0 & 42 & 420 &    0.0 & 57 & 228 &    0.0 \\\hline
 (1000,200) & 100  &    0.3 & 62 &    0.1 & 53 & 265 &    0.0 & 55 & 550 &    0.1 & 79 & 316 &    0.0 \\\hline
 (2000,200) & 100  &    0.6 & 1000 &    1.4 & 44 & 220 &    0.1 & 1000 & 10000 &    4.5 & 65 & 260 &    0.1 \\\hline
 (5000,100) & 100  &    0.7 & 31 &    0.0 & 30 & 150 &    0.1 & 30 & 300 &    0.2 & 36 & 144 &    0.1 \\\hline
 (5000,200) & 100  &    1.4 & 34 &    0.1 & 33 & 165 &    0.4 & 34 & 340 &    0.7 & 45 & 180 &    0.2 \\\hline
 (8000,100) & 100  &    1.2 & 32 &    0.1 & 30 & 150 &    0.3 & 36 & 360 &    0.6 & 52 & 208 &    0.2 \\\hline
 (8000,200) & 100  &    2.4 & 35 &    0.1 & 34 & 170 &    0.6 & 31 & 310 &    1.0 & 45 & 180 &    0.4 \\\hline
 \end{tabular}
\end{center}
\end{table}
}

{\small
\begin{table}[ht]
\begin{center}\caption{{Numerical Results for Lasso problem}}\label{tab:lasso-2}
\begin{tabular}{| l | ll | ll | lll | lll | lll |}\hline
         & \multicolumn{2}{|c|}{ISTA} & \multicolumn{2}{|c|}{ADMM} & \multicolumn{3}{|c|}{ADMM-5} & \multicolumn{3}{|c|}{ADMM-10} & \multicolumn{3}{|c|}{EGADM} \\\hline
$(m,n)$ & iter & cpu & iter & cpu & iter & mvm & cpu & iter & mvm & cpu & iter & mvm & cpu \\\hline
\multicolumn{14}{|c|}{$\gamma=0.5$} \\\hline
 (100,1000) & 100 &    0.1 & 51 &    2.4 & 67 & 335 &    0.0 & 54 & 540 &    0.0 & 202 & 808 &    0.0 \\\hline
 (100,2000) & 100  &    0.3 & 51 &   14.6 & 67 & 335 &    0.1 & 54 & 540 &    0.1 & 202 & 808 &    0.1 \\\hline
 (100,5000) & 100 &    0.7 & 51 &  206.9 & 67 & 335 &    0.4 & 54 & 540 &    0.5 & 202 & 808 &    0.5 \\\hline
 (100,8000) & 100 &    1.2 & 51 &  793.9 & 67 & 335 &    0.7 & 54 & 540 &    0.9 & 202 & 808 &    0.9 \\\hline
 (1000,100) & 100 &    0.1 & 36 &    0.0 & 36 & 180 &    0.0 & 36 & 360 &    0.0 & 95 & 380 &    0.0 \\\hline
 (1000,200) & 100  &    0.3 & 46 &    0.1 & 42 & 210 &    0.0 & 58 & 580 &    0.1 & 128 & 512 &    0.0 \\\hline
 (2000,200) & 100  &    0.6 & 1000 &    1.3 & 52 & 260 &    0.1 & 41 & 410 &    0.2 & 126 & 504 &    0.2 \\\hline
 (5000,100) & 100  &    0.7 & 33 &    0.0 & 31 & 155 &    0.1 & 31 & 310 &    0.3 & 65 & 260 &    0.1 \\\hline
 (5000,200) & 100 &    1.4 & 33 &    0.1 & 34 & 170 &    0.4 & 35 & 350 &    0.7 & 75 & 300 &    0.4 \\\hline
 (8000,100) & 100  &    1.2 & 35 &    0.1 & 37 & 185 &    0.3 & 49 & 490 &    0.8 & 71 & 284 &    0.3 \\\hline
 (8000,200) & 100 &    2.2 & 34 &    0.1 & 37 & 185 &    0.6 & 33 & 330 &    1.0 & 72 & 288 &    0.5 \\\hline
\multicolumn{14}{|c|}{$\gamma=0.1$} \\\hline
 (100,1000) & 100 &   0.1 & 12 &    0.5 & 206 & 1030 &    0.1 & 106 & 1060 &    0.1 & 1000 & 4000 &    0.2 \\\hline
 (100,2000) & 100 &    0.3 & 12 &    3.6 & 206 & 1030 &    0.2 & 106 & 1060 &    0.1 & 1000 & 4000 &    0.4 \\\hline
 (100,5000) & 100 &    0.7 & 12 &   46.1 & 206 & 1030 &    1.1 & 106 & 1060 &    1.0 & 1000 & 4000 &    2.5 \\\hline
 (100,8000) & 100 &   1.2 & 12 &  192.3 & 206 & 1030 &    2.1 & 106 & 1060 &    1.9 & 1000 & 4000 &    4.4 \\\hline
 (1000,100) & 100 &   0.1 & 122 &    0.0 & 145 & 725 &    0.0 & 109 & 1090 &    0.1 & 592 & 2368 &    0.1 \\\hline
 (1000,200) & 100 &     0.3 & 114 &    0.1 & 207 & 1035 &    0.1 & 127 & 1270 &    0.1 & 935 & 3740 &    0.2 \\\hline
 (2000,200) & 100 &    0.6 & 118 &    0.2 & 173 & 865 &    0.5 & 137 & 1370 &    0.6 & 660 & 2640 &    0.8 \\\hline
 (5000,100) & 100 &     0.8 & 108 &    0.1 & 112 & 560 &    0.6 & 118 & 1180 &    1.0 & 363 & 1452 &    0.8 \\\hline
 (5000,200) & 100 &     1.4 & 107 &    0.2 & 121 & 605 &    1.3 & 115 & 1150 &    2.3 & 480 & 1920 &    2.4 \\\hline
 (8000,100) & 100 &     1.2 & 117 &    0.1 & 138 & 690 &    1.2 & 144 & 1440 &    2.3 & 360 & 1440 &    1.4 \\\hline
 (8000,200) & 100 &   2.4 & 110 &    0.3 & 128 & 640 &    2.1 & 133 & 1330 &    4.3 & 425 & 1700 &    3.2 \\\hline
\end{tabular}
\end{center}
\end{table}
}

From Tables \ref{tab:lasso-1} and \ref{tab:lasso-2} we have the following observations. When $\gamma=1.0$ and $m<n$, the number of iterations of EGADM, ADMM, ADMM-5 and ADMM-10 are similar with ISTA, while ADMM is more costly because it needs to solve an $n\times n$ linear system in each iteration. When $\gamma=1.0$ and $m>n$, ADMM is better than other solvers, because now the linear system is in a relatively small size. Moreover, EGADM also performs well when $m>n$. It needs less iteration number than ISTA (except one instance). However, ADMM-5 and ADMM-10 do not perform well, because the step size $\gamma$ is too large.
When $\gamma$ is modest, i.e., $\gamma=0.8$ and $\gamma=0.5$, ADMM, ADMM-5 and ADMM-10 usually perform better than ISTA in terms of number of iterations. EGADM is comparable to ISTA in these two cases in the sense that EGADM needs more iterations than ISTA when $m<n$ and less iterations when $m>n$. When $\gamma=0.1$, ADMM-5, ADMM-10 and EGADM are worse than ISTA, but we emphasize here that ISTA always uses a large step size $\gamma=1$.

Based on these comparison results, we may conclude as follows.
First, ADMM is usually better than ISTA and EGADM when the two subproblems can be easily solved. Second, ADMM is not very sensitive to the selection of the penalty parameter $\gamma$, while ISTA and EGADM depends on it because these two methods are more like gradient method and need to choose step size properly. In the case that one subproblem is not easy to solve in ADMM, running gradient descent method for several iterations to get an approximate solution to the subproblem usually works. However, choosing the step size $\gamma$ is very crucial to the performance. Furthermore, we emphasize again that ISTA only works for problems without linear constraints, and ADMM works when both subproblems can be solved easily. In contrast, EGADM provides something in the middle: it works when linear constraints are present and one subproblem is not easy to solve.

\begin{remark}
We remark here that as shown in many references, the current state-of-the-art algorithm for solving the lasso problem \eqref{prob:lasso}, is the randomized coordinate descent method (see, e,g, \cite{richtarik-mpa-2014,Richtarik-Takac-2012-parallel,Zhang-Tong-MPA-2015}). In our numerical comparison conducted in this subsection, we only compared our EGADM with ISTA and ADMM. We did not compare with the randomized coordinate descent method because EGADM, ISTA and ADMM are all batch type methods and we want to focus the comparison among the efficient batch methods for lasso problem. Moreover, stochastic ADMM as studied in \cite{Suzuki-icml-2014} can solve the dual problem of \eqref{prob:min-sum-2}. It is suitable for solving problems where $f$ is the empirical loss function in machine learning or data fitting problems. For example, it can be applied to solve the lasso problem \eqref{prob:lasso} or the so-called graph lasso problem. However, the method in \cite{Suzuki-icml-2014} assumes that the proximal mappings of both $f$ and $g$ are easily computable. As a result, it is not suitable for solving the fused logistic regression problem \eqref{fused-logistic-reg}.
\end{remark}

\begin{remark}
Since our proposed EGADM method is a variant of ADMM, it is capable of solving problems at least as large as that can be handled by ADMM. Moreover, the main computational effort in each iteration of EGADM is computing the proximal mapping of function $f$ and computing the gradient of function $g$. EGADM is efficient as long as these computations can be done relatively easily. From Tables  \ref{tab:fused-logistic-reg}, \ref{tab:lasso-1} and \ref{tab:lasso-2} we see that EGADM is efficient for at least medium-sized problems. For even larger problems, we may implement some parallel or randomized versions of EGADM which will be left as future research topics.
\end{remark}

\section{Concluding Remarks}\label{sec:conclusion}

In this paper, we proposed a new alternating direction method based on extragradient for solving convex minimization problems with the objective function being the sum of two convex functions. The proposed method applies to the situation where only one of the involved functions has easy proximal mapping, while the other function is only known to be smooth. Under the assumption that the smooth function has a Lipschitz continuous gradient, we proved that the proposed method finds an $\epsilon$-optimal solution within $O(1/\epsilon)$ iterations. We used the lasso problem to illustrate the performance of the proposed method and compared its performance with some existing solvers for this problem. We also proposed a new statistical model, namely fused logistic regression, that can preserve the natural ordering of the features in logistic regression. Preliminary numerical results showed that this new model is preferable than the sparse logistic regression model when there exists natural ordering in the features. The numerical results also showed that our extragradient-based ADM can solve large-scale fused logistic regression model efficiently.

It is noted that we only considered problems with two block variables in this paper. If there are $N$-block variables with $N\geq 3$, the proposed extragradient-based ADM is still applicable. In particular, if the first $N-1$ functions have easy proximal mappings and the last one does not but is smooth, we can apply the multi-block ADMM to solve this problem, but replacing the minimization for the augmented Lagrangian function with respect to the last block variable by an extragradient step. However, the convergence properties of this algorithm are currently not known, and this will be a topic for the future research. One simplification of EGADM is {\em GADM}; that is, replacing the extra-gradient steps by a single gradient step. The method was proposed in an earlier version of this paper, whose numerical performance was found to be comparable with EGADM; however, the convergence status of GADM was unknown to us. Very recently, \cite{Gao-Jiang-Zhang-2014} resolved the issue and proved that the iteration complexity of GADM remains $O(1/\epsilon)$.

\section*{Acknowledgements}

We thank Lingzhou Xue and Hui Zou for fruitful discussions on logistic regression and fused lasso. We are also grateful to two anonymous referees for their constructive comments that have helped improve the presentation of this paper greatly.

\newpage

\bibliographystyle{plain}
\bibliography{All}

\end{document}